\documentclass{amsart}
\usepackage{amssymb}
\usepackage{amsfonts}

\newtheorem{theorem}{Theorem}
\theoremstyle{plain}

\newtheorem{corollary}{Corollary}

\newtheorem{lemma}{Lemma}

\numberwithin{equation}{section}
\input{tcilatex}

\begin{document}
\title[Inequalities for Quadratic Operator Perspective]{Inequalities for
Quadratic Operator Perspective of Convex Functions and Bounded Linear
Operators in Hilbert Spaces}
\author{S. S. Dragomir$^{1,2}$}
\address{$^{1}$Mathematics, College of Engineering \& Science\\
Victoria University, PO Box 14428\\
Melbourne City, MC 8001, Australia.}
\email{sever.dragomir@vu.edu.au}
\urladdr{http://rgmia.org/dragomir}
\address{$^{2}$DST-NRF Centre of Excellence \\
in the Mathematical and Statistical Sciences, School of Computer Science \&
Applied Mathematics, University of the Witwatersrand, Private Bag 3,
Johannesburg 2050, South Africa}
\subjclass{47A63, 47A30, 15A60, 26D15, 26D10}
\keywords{Operator perspective, Convex functions, Operator inequalities,
Arithmetic mean-Geometric mean operator inequality, Relative operator
entropy.}

\begin{abstract}
In this paper we introduce the concept of \textit{quadratic operator
perspective} for a continuos function $\Phi $ defined on the positive
semi-axis of real numbers, the invertible operator $T$ and operator $V$ on a
Hilbert space by \ 
\begin{equation*}
\circledcirc _{\Phi }\left( V,T\right) :=T^{\ast }\Phi \left( \left\vert
VT^{-1}\right\vert ^{2}\right) T.
\end{equation*}%
This generalize the \textit{quadratic} \textit{weighted operator geometric
mean }of $\left( T,V\right) $ defined by%
\begin{equation*}
T\circledS _{\nu }V:=\left\vert \left\vert VT^{-1}\right\vert ^{\nu
}T\right\vert ^{2}
\end{equation*}%
for $\nu \in \left[ 0,1\right] $ and the \textit{quadratic} \textit{relative
operator entropy defined }by%
\begin{equation*}
\odot \left( T|V\right) :=T^{\ast }\ln \left( \left\vert VT^{-1}\right\vert
^{2}\right) T.
\end{equation*}%
Some inequalities for this perspective of convex functions are established.
Applications for quadratic weighted operator geometric mean and quadratic
relative operator entropy are also provided.
\end{abstract}

\maketitle

\section{Introduction}

If $\Phi :I\rightarrow \mathbb{R}$ is a convex function on the real interval 
$I$ and $T$ is a selfadjoint operator on the complex Hilbert space $\left(
H;\left\langle \cdot ,\cdot \right\rangle \right) $ with the spectrum $%
\limfunc{Sp}\left( T\right) \subset \mathring{I},$ the interior of $I,$ then
we have the following Jensen's type inequality 
\begin{equation}
\left\langle \Phi \left( T\right) x,x\right\rangle \geq \Phi \left(
\left\langle Tx,x\right\rangle \right)  \label{Jen}
\end{equation}%
for any $x\in H$ with $\left\Vert x\right\Vert =1.$

For various Jensen's type inequalities for functions of selfadjoint
operators, see the recent monograph \cite{SSDJen} and the references therein.

In the recent paper \cite{SSDPC} we showed amongst others that if $A$ is a
positive invertible operator and $B$ is a selfadjoint operator such that $%
\limfunc{Sp}\left( A^{-1/2}BA^{-1/2}\right) \subset \mathring{I},$ then 
\begin{equation}
\frac{\left\langle A^{1/2}\Phi \left( A^{-1/2}BA^{-1/2}\right)
A^{1/2}x,x\right\rangle }{\left\langle Ax,x\right\rangle }\geq \Phi \left( 
\frac{\left\langle Bx,x\right\rangle }{\left\langle Ax,x\right\rangle }%
\right) ,  \label{Jen1}
\end{equation}%
for any $x\in H,$ $x\neq 0.$ This result can be reformulated in terms of
perspective as follows.

Let $\Phi $ be a continuous function defined on the interval $I$ of real
numbers, $B$ a selfadjoint operator on the Hilbert space $H$ and $A$ a
positive invertible operator on $H.$ Assume that the spectrum $\limfunc{Sp}%
\left( A^{-1/2}BA^{-1/2}\right) \subset \mathring{I}.$ Then by using the
continuous functional calculus, we can define the \textit{perspective} $%
\mathcal{P}_{\Phi }\left( B,A\right) $ by setting 
\begin{equation*}
\mathcal{P}_{\Phi }\left( B,A\right) :=A^{1/2}\Phi \left(
A^{-1/2}BA^{-1/2}\right) A^{1/2}.
\end{equation*}%
If $A$ and $B$ are commutative, then 
\begin{equation*}
\mathcal{P}_{\Phi }\left( B,A\right) =A\Phi \left( BA^{-1}\right)
\end{equation*}%
provided $\limfunc{Sp}\left( BA^{-1}\right) \subset \mathring{I}.$

By using the perspective notation, we have by (\ref{Jen1}) that 
\begin{equation}
\frac{\left\langle \mathcal{P}_{\Phi }\left( B,A\right) x,x\right\rangle }{%
\left\langle Ax,x\right\rangle }\geq \Phi \left( \frac{\left\langle
Bx,x\right\rangle }{\left\langle Ax,x\right\rangle }\right) ,  \label{Jen2}
\end{equation}%
for any $x\in H$ with $\left\Vert x\right\Vert =1.$

It is well known that (see \cite{E} and \cite{ENG} or \cite{EH}), if $\Phi $
is an \textit{operator convex function} defined in the positive half-line,
then the mapping 
\begin{equation*}
\left( B,A\right) \rightarrow \mathcal{P}_{\Phi }\left( B,A\right)
\end{equation*}%
\textit{defined in pairs of positive definite operators, is convex.}

Assume that $A,$ $B$ are positive operators on a complex Hilbert space $%
\left( H,\left\langle \cdot ,\cdot \right\rangle \right) .$ The \textit{%
weighted operator arithmetic mean }for the pair $\left( A,B\right) $ is
defined by 
\begin{equation*}
A\nabla _{\nu }B:=\left( 1-\nu \right) A+\nu B.
\end{equation*}%
In 1980, Kubo \& Ando, \cite{KA} introduced the \textit{weighted operator
geometric mean }for the pair $\left( A,B\right) $ with $A$ positive and
invertible and $B$\textit{\ }positive by%
\begin{equation*}
A\sharp _{\nu }B:=A^{1/2}\left( A^{-1/2}BA^{-1/2}\right) ^{\nu }A^{1/2}.
\end{equation*}%
If $A,$ $B$ are positive invertible operators then we can also consider the 
\textit{weighted operator harmonic mean }defined by (see for instance \cite%
{KA}) 
\begin{equation*}
A!_{\nu }B:=\left( \left( 1-\nu \right) A^{-1}+\nu B^{-1}\right) ^{-1}.
\end{equation*}

We have the following fundamental operator means inequalities, or Young's
inequalities%
\begin{equation}
A!_{\nu }B\leq A\sharp _{\nu }B\leq A\nabla _{\nu }B,\text{ }\nu \in \left[
0,1\right]  \label{KA}
\end{equation}%
for any $A,$ $B$ positive invertible operators. For $\nu =\frac{1}{2},$ we
denote the above means by $A\nabla B,$ $A\sharp B$ and $A!B.$

For recent results on operator Young inequality see \cite{F}-\cite{KM2}, 
\cite{LWZ} and \cite{T}-\cite{ZSF}.

We denote by $\mathcal{B}^{-1}\left( H\right) $ the class of all bounded
linear invertible operators on $H.$ For $T\in \mathcal{B}^{-1}\left(
H\right) $ and $V\in \mathcal{B}\left( H\right) $ we define the \textit{%
quadratic} \textit{weighted operator geometric mean }of $\left( T,V\right) $
by \cite{SSDQ} 
\begin{equation}
T\circledS _{\nu }V:=\left\vert \left\vert VT^{-1}\right\vert ^{\nu
}T\right\vert ^{2}  \label{S}
\end{equation}%
for $\nu \geq 0.$ For $V\in \mathcal{B}^{-1}\left( H\right) $ we can also
extend the definition (\ref{S}) for $\nu <0.$

By the definition of operator modulus, i.e., we recall that $\left\vert
U\right\vert :=\sqrt{U^{\ast }U},$ $U\in \mathcal{B}\left( H\right) ,$ we
also have%
\begin{equation}
T\circledS _{\nu }V=T^{\ast }\left\vert VT^{-1}\right\vert ^{2\nu }T=T^{\ast
}\left( \left( T^{\ast }\right) ^{-1}V^{\ast }VT^{-1}\right) ^{\nu }T
\label{e.1.5}
\end{equation}%
for any $T\in \mathcal{B}^{-1}\left( H\right) $ and $V\in \mathcal{B}\left(
H\right) .$ For $\nu =\frac{1}{2}$ we denote 
\begin{equation*}
T\circledS V:=\left\vert \left\vert VT^{-1}\right\vert ^{1/2}T\right\vert
^{2}=T^{\ast }\left\vert VT^{-1}\right\vert T=T^{\ast }\left( \left( T^{\ast
}\right) ^{-1}V^{\ast }VT^{-1}\right) ^{1/2}T.
\end{equation*}

If we take in (\ref{S}) $T=A^{1/2}\in \mathcal{B}^{-1}\left( H\right) $ and $%
V=B^{1/2}$ with $A$ a positive invertible operator and $B$ a nonnegative
operator, then we get 
\begin{equation}
A^{1/2}\circledS _{\nu }B^{1/2}=A\sharp _{\nu }B\text{ for }\nu \in \left[
0,1\right] .  \label{e.1.4}
\end{equation}

We have the following fundamental inequalities \cite{SSDQ}:%
\begin{equation}
\left\vert T\right\vert ^{2}\nabla _{\nu }\left\vert V\right\vert ^{2}\geq
T\circledS _{\nu }V  \label{e.1.6}
\end{equation}%
for any $T\in \mathcal{B}^{-1}\left( H\right) $, $V\in \mathcal{B}\left(
H\right) $ and $\nu \in \left[ 0,1\right] .$ If $T,$ $V\in \mathcal{B}%
^{-1}\left( H\right) ,$ then for $\nu \in \left[ 0,1\right] $ we also have%
\begin{equation}
T\circledS _{\nu }V\geq \left\vert T\right\vert ^{2}!_{\nu }\left\vert
V\right\vert ^{2}.  \label{e.1.7}
\end{equation}%
In particular, we have%
\begin{equation}
\left\vert T\right\vert ^{2}\nabla \left\vert V\right\vert ^{2}\geq
T\circledS V\geq \left\vert T\right\vert ^{2}!\left\vert V\right\vert ^{2}
\label{e.1.8}
\end{equation}%
for $T,$ $V\in \mathcal{B}^{-1}\left( H\right) .$

We have the following identities \cite{SSDK} as well%
\begin{equation}
\left( T\circledS _{\nu }V\right) ^{-1}=\left( T^{\ast }\right)
^{-1}\circledS _{\nu }\left( V^{\ast }\right) ^{-1}\text{ and }T\circledS
_{1-t}V=V\circledS _{t}T  \label{e.1.9}
\end{equation}%
for any $T,$ $V\in \mathcal{B}^{-1}\left( H\right) $ and $\nu \in \left[ 0,1%
\right] .$

Kamei and Fujii \cite{FK1}, \cite{FK2} defined the \textit{relative operator
entropy} $S\left( A|B\right) ,$ for positive invertible operators $A$ and $%
B, $ by 
\begin{equation}
S\left( A|B\right) :=A^{\frac{1}{2}}\left( \ln \left( A^{-\frac{1}{2}}BA^{-%
\frac{1}{2}}\right) \right) A^{\frac{1}{2}},  \label{e.1}
\end{equation}%
which is a relative version of the operator entropy considered by
Nakamura-Umegaki \cite{NU}.

For some recent results on relative operator entropy see \cite{SSDE1}-\cite%
{SSDE2}, \cite{K}-\cite{KN} and \cite{MMM}-\cite{Ni}.

Consider the scalar function $T_{t}:\left( 0,\infty \right) \rightarrow 
\mathbb{R}$ defined for $t\neq 0$ by%
\begin{equation}
T_{t}\left( x\right) :=\frac{x^{t}-1}{t}.  \label{Tt}
\end{equation}%
We have%
\begin{equation}
T_{-t}\left( x\right) =\frac{1-x^{-t}}{t}=\frac{x^{t}-1}{tx^{t}}=T_{t}\left(
x\right) x^{-t}.  \label{e.1.6.a}
\end{equation}

For $T,$ $V\in \mathcal{B}^{-1}\left( H\right) $ and $t>0$ we define the 
\textit{quadratic Tsallis relative operator entropy} by \cite{SSDRE} 
\begin{align}
\circledcirc _{t}\left( T|V\right) & :=T^{\ast }T_{t}\left( \left\vert
VT^{-1}\right\vert ^{2}\right) T=T^{\ast }\frac{\left( \left\vert
VT^{-1}\right\vert ^{2}\right) ^{t}-1}{t}T  \label{TQ} \\
& =\frac{T\circledS _{t}V-\left\vert T\right\vert ^{2}}{t}=\frac{\left\vert
\left\vert VT^{-1}\right\vert ^{t}T\right\vert ^{2}-\left\vert T\right\vert
^{2}}{t}  \notag
\end{align}%
and the \textit{quadratic} \textit{relative operator entropy }by \cite{SSDRE}%
\begin{equation}
\odot \left( T|V\right) :=T^{\ast }\ln \left( \left\vert VT^{-1}\right\vert
^{2}\right) T.  \label{QR}
\end{equation}%
We also have for $t>0$ and $T,$ $V\in \mathcal{B}^{-1}\left( H\right) $ that%
\begin{equation}
\circledcirc _{-t}\left( T|V\right) =T^{\ast }T_{-t}\left( \left\vert
VT^{-1}\right\vert ^{2}\right) T=\circledcirc _{t}\left( T|V\right) \left(
T\circledS _{t}V\right) ^{-1}\left\vert T\right\vert ^{2}.  \label{-TQ}
\end{equation}

We observe that for $T=A^{1/2}\in \mathcal{B}^{-1}\left( H\right) $ and $%
V=B^{1/2}\in \mathcal{B}^{-1}\left( H\right) $ we get the equalities%
\begin{equation*}
\circledcirc _{t}\left( A^{1/2}|B^{1/2}\right) =T_{t}\left( A|B\right) :=%
\frac{A\sharp _{\nu }B-A}{t}\text{ }
\end{equation*}%
and 
\begin{equation*}
\odot \left( A^{1/2}|B^{1/2}\right) =S\left( A|B\right) ,
\end{equation*}%
that show the connection between the extended Tsallis and relative entropies
with the classical concepts defined for positive operators.

The following fundamental inequalities may be stated \cite{SSDRE}:%
\begin{equation}
\circledcirc _{-t}\left( T|V\right) \leq \odot \left( T|V\right) \leq
\circledcirc _{t}\left( T|V\right)   \label{e.1.10}
\end{equation}%
for any $T,$ $V\in \mathcal{B}^{-1}\left( H\right) $ and $t>0.$

In particular, we have%
\begin{equation}
\left( 1_{H}-\left\vert T\right\vert ^{2}\left\vert V\right\vert
^{-2}\right) \left\vert T\right\vert ^{2}\leq \odot \left( T|V\right) \leq
\left\vert V\right\vert ^{2}-\left\vert T\right\vert ^{2},  \label{e.1.11}
\end{equation}%
and%
\begin{equation}
2\left( 1_{H}-\left\vert T\right\vert ^{2}\left( T\circledS V\right)
^{-1}\right) \left\vert T\right\vert ^{2}\leq \odot \left( T|V\right) \leq
2\left( T\circledS V-\left\vert T\right\vert ^{2}\right)  \label{e.2.3}
\end{equation}%
for any $T,$ $V\in \mathcal{B}^{-1}\left( H\right) .$

Let $T\in \mathcal{B}^{-1}\left( H\right) $, $V\in \mathcal{B}\left(
H\right) $ and $I$ an interval of nonnegative numbers. Assume that $\limfunc{%
Sp}\left( \left\vert VT^{-1}\right\vert ^{2}\right) \subset \mathring{I}$
and $\Phi $ is a continuous function defined on the interval $I.$ Then by
using the continuous functional calculus for selfadjoint operators, we can
define the \textit{quadratic} \textit{operator} \textit{perspective }of%
\textit{\ }$T$, $V$ and $\Phi $ by 
\begin{equation}
\circledcirc _{\Phi }\left( V,T\right) :=T^{\ast }\Phi \left( \left\vert
VT^{-1}\right\vert ^{2}\right) T.  \label{Def}
\end{equation}%
If we take in this definition $\Phi \left( x\right) =x^{\nu },$ $x>0,$ $\nu
\neq 0,$ then we recapture the definition of quadratic weighted operator
geometric mean\textit{, }for\textit{\ }$\Phi \left( x\right) =\frac{x^{t}-1}{%
t},$ $t\neq 0,$ $x>0,$ the definition of quadratic Tsallis relative operator
entropy\textit{\ }and for $\Phi \left( x\right) =\ln x,$ $x>0$ the
definition of quadratic relative operator entropy\textit{.}

Motivated by the above facts, we establish in this paper some upper and
lower bounds for the quadratic operator perspective and apply them for the
quadratic operator entropy and geometric mean defined above.

\section{Operator Inequalities for Quadratic Perspectives}

Suppose that $I$ is an interval of real numbers with interior $\mathring{I}$
and $\Phi :I\rightarrow \mathbb{R}$ is a convex function on $I$. Then $\Phi $
is continuous on $\mathring{I}$ and has finite left and right derivatives at
each point of $\mathring{I}$. Moreover, if $t,$ $s\in \mathring{I}$ and $%
t<s, $ then $\Phi _{-}^{\prime }\left( t\right) \leq \Phi _{+}^{\prime
}\left( t\right) \leq \Phi _{-}^{\prime }\left( s\right) \leq \Phi
_{+}^{\prime }\left( s\right) $ which shows that both $\Phi _{-}^{\prime }$
and $\Phi _{+}^{\prime }$ are nondecreasing function on $\mathring{I}$. It
is also known that a convex function must be differentiable except for at
most countably many points.

For a convex function $\Phi :I\rightarrow \mathbb{R}$, the subdifferential
of $\Phi $ denoted by $\partial \Phi $ is the set of all functions $\varphi
:I\rightarrow \left[ -\infty ,\infty \right] $ such that $\varphi \left( 
\mathring{I}\right) \subset \mathbb{R}$ and 
\begin{equation}
\Phi \left( t\right) \geq \Phi \left( a\right) +\left( t-a\right) \varphi
\left( a\right) \text{ for any }t,\text{ }a\in I.  \label{G}
\end{equation}

It is also well known that if $\Phi $ is convex on $I,$ then $\partial \Phi $
is nonempty, $\Phi _{-}^{\prime }$, $\Phi _{+}^{\prime }\in \partial \Phi $
and if $\varphi \in \partial \Phi $, then 
\begin{equation*}
\Phi _{-}^{\prime }\left( t\right) \leq \varphi \left( t\right) \leq \Phi
_{+}^{\prime }\left( t\right) \text{ for any }t\in \mathring{I}\text{.}
\end{equation*}%
In particular, $\varphi $ is a nondecreasing function.

If $\Phi $ is differentiable and convex on $\mathring{I}$, then $\partial
\Phi =\left\{ \Phi ^{\prime }\right\} .$

We need the following simple fact, see also \cite{SSDQ}:

\begin{lemma}
\label{l.4.1}Let $T,$ $V\in \mathcal{B}^{-1}\left( H\right) $ and $%
0<m<M<\infty .$ Then the following statements are equivalent:

(i) The inequality 
\begin{equation}
m\left\Vert Tx\right\Vert \leq \left\Vert Vx\right\Vert \leq M\left\Vert
Tx\right\Vert  \label{mM}
\end{equation}%
holds for any $x\in H;$

(ii) We have the operator inequality%
\begin{equation}
m1_{H}\leq \left\vert VT^{-1}\right\vert \leq M1_{H}.  \label{mM'}
\end{equation}
\end{lemma}

\begin{proof}
The inequality (\ref{mM}) is equivalent to 
\begin{equation*}
m^{2}\left\Vert Tx\right\Vert ^{2}\leq \left\Vert Vx\right\Vert ^{2}\leq
M^{2}\left\Vert Tx\right\Vert ^{2}
\end{equation*}%
for any $x\in H,$ namely%
\begin{equation*}
m^{2}\left\langle T^{\ast }Tx,x\right\rangle \leq \left\langle V^{\ast
}Vx,x\right\rangle \leq M^{2}\left\langle T^{\ast }Tx,x\right\rangle
\end{equation*}%
for any $x\in H,$ which can be written in the operator order as 
\begin{equation*}
m^{2}T^{\ast }T\leq V^{\ast }V\leq M^{2}T^{\ast }T.
\end{equation*}%
Since $T\in \mathcal{B}^{-1}\left( H\right) $, then this inequality is
equivalent to 
\begin{equation*}
m^{2}1_{H}\leq \left( T^{-1}\right) ^{\ast }V^{\ast }VT^{-1}\leq M^{2}1_{H},
\end{equation*}%
namely 
\begin{equation*}
m^{2}1_{H}\leq \left\vert VT^{-1}\right\vert ^{2}\leq M^{2}1_{H},
\end{equation*}%
which in its turn is equivalent to (\ref{mM'}).
\end{proof}

We have:

\begin{theorem}
\label{t.2.1}Let $\Phi :I\rightarrow \mathbb{R}$ be a convex function on the
interval of positive numbers $I$, $T,$ $V\in \mathcal{B}^{-1}\left( H\right) 
$ such that there exists the positive numbers $m<M$ with $\left[ m^{2},M^{2}%
\right] \subset \mathring{I}$ satisfying either the condition (\ref{mM}),
or, equivalently, the condition (\ref{mM'}). Then for any $\varphi \in
\partial \Phi $ and any $t\in \mathring{I}$ we have%
\begin{equation}
\circledcirc _{\Phi }\left( V,T\right) \geq \Phi \left( t\right) \left\vert
T\right\vert ^{2}+\varphi \left( t\right) \left( \left\vert V\right\vert
^{2}-t\left\vert T\right\vert ^{2}\right) .  \label{e.2.1}
\end{equation}%
In particular, 
\begin{align}
\circledcirc _{\Phi }\left( V,T\right) & \geq \Phi \left( \frac{m^{2}+M^{2}}{%
2}\right) \left\vert T\right\vert ^{2}  \label{e.2.1.a} \\
& +\varphi \left( \frac{m^{2}+M^{2}}{2}\right) \left( \left\vert
V\right\vert ^{2}-\frac{m^{2}+M^{2}}{2}\left\vert T\right\vert ^{2}\right) .
\notag
\end{align}
\end{theorem}

\begin{proof}
From (\ref{G}) we have%
\begin{equation}
\Phi \left( s\right) \geq \Phi \left( t\right) +\left( s-t\right) \varphi
\left( t\right)  \label{e.2.2}
\end{equation}%
for any $s\in \left[ m^{2},M^{2}\right] $ and $t\in \mathring{I}$.

Using the continuous functional calculus for a selfadjoint operator $X$ with 
$\limfunc{Sp}\left( X\right) \subseteq \left[ m^{2},M^{2}\right] \subset 
\mathring{I}$ we have from (\ref{e.2.2}) in the operator order that 
\begin{equation}
\Phi \left( X\right) \geq \Phi \left( t\right) 1_{H}+\varphi \left( t\right)
\left( X-t1_{H}\right)  \label{e.2.3.a}
\end{equation}%
for any $t\in \mathring{I}$.

Now, if we take $X=\left\vert VT^{-1}\right\vert ^{2}$ in (\ref{e.2.3.a}),
then we get%
\begin{equation}
\Phi \left( \left\vert VT^{-1}\right\vert ^{2}\right) \geq \Phi \left(
t\right) 1_{H}+\varphi \left( t\right) \left( \left\vert VT^{-1}\right\vert
^{2}-t1_{H}\right)  \label{e.2.4}
\end{equation}%
for any $t\in \mathring{I}$.

It is well know that, if $P\geq 0$ then by multiplying at left with $T^{\ast
}$ and at right with $T,$ where $T\in \mathcal{B}\left( H\right) $ we have
that $T^{\ast }PT\geq 0.$ If $A,$ $B$ are selfadjoint operators with $A\geq
B $ then for any $T\in \mathcal{B}\left( H\right) $ we have $T^{\ast }AT\geq
T^{\ast }BT.$

So, if we multiply (\ref{e.2.4}) at left with $T^{\ast }$ and at right with $%
T,$ then we get%
\begin{align*}
T^{\ast }\Phi \left( \left\vert VT^{-1}\right\vert ^{2}\right) T& \geq \Phi
\left( t\right) \left\vert T\right\vert ^{2}+\varphi \left( t\right) T^{\ast
}\left( \left\vert VT^{-1}\right\vert ^{2}-t1_{H}\right) T \\
& =\Phi \left( t\right) \left\vert T\right\vert ^{2}+\varphi \left( t\right)
T^{\ast }\left( \left( T^{\ast }\right) ^{-1}V^{\ast }VT^{-1}-t1_{H}\right) T
\\
& =\Phi \left( t\right) \left\vert T\right\vert ^{2}+\varphi \left( t\right)
\left( \left\vert V\right\vert ^{2}-t\left\vert T\right\vert ^{2}\right)
\end{align*}%
for any $t\in \mathring{I}$, which proves the desired inequality (\ref{e.2.1}%
).
\end{proof}

\begin{corollary}
\label{c.2.1}With the assumptions of Theorem \ref{t.2.1}, we have for any $%
x\in H\setminus \left\{ 0\right\} $ that%
\begin{equation}
\circledcirc _{\Phi }\left( V,T\right) \geq \Phi \left( \frac{\left\Vert
Vx\right\Vert ^{2}}{\left\Vert Tx\right\Vert ^{2}}\right) \left\vert
T\right\vert ^{2}+\varphi \left( \frac{\left\Vert Vx\right\Vert ^{2}}{%
\left\Vert Tx\right\Vert ^{2}}\right) \left( \left\vert V\right\vert ^{2}-%
\frac{\left\Vert Vx\right\Vert ^{2}}{\left\Vert Tx\right\Vert ^{2}}%
\left\vert T\right\vert ^{2}\right)  \label{e.2.6}
\end{equation}%
in the operator order of $\mathcal{B}\left( H\right) .$

In particular, we have the Jensen's type inequality 
\begin{equation}
\frac{\left\langle \circledcirc _{\Phi }\left( V,T\right) x,x\right\rangle }{%
\left\Vert Tx\right\Vert ^{2}}\geq \Phi \left( \frac{\left\Vert
Vx\right\Vert ^{2}}{\left\Vert Tx\right\Vert ^{2}}\right)  \label{e.2.7}
\end{equation}%
$x\in H\setminus \left\{ 0\right\} .$
\end{corollary}

\begin{proof}
For $x\in H\setminus \left\{ 0\right\} $ we have 
\begin{align*}
t_{A,B}& =\frac{\left\Vert Vx\right\Vert ^{2}}{\left\Vert Tx\right\Vert ^{2}}%
=\frac{\left\langle \left\vert V\right\vert ^{2}x,x\right\rangle }{%
\left\langle \left\vert T\right\vert ^{2}x,x\right\rangle }=\frac{%
\left\langle T^{\ast }\left( T^{\ast }\right) ^{-1}V^{\ast
}VT^{-1}Tx,x\right\rangle }{\left\langle Tx,Tx\right\rangle } \\
& =\frac{\left\langle \left( \left( T^{\ast }\right) ^{-1}V^{\ast
}VT^{-1}\right) Tx,Tx\right\rangle }{\left\langle Tx,Tx\right\rangle }=\frac{%
\left\langle \left( \left( T^{\ast }\right) ^{-1}V^{\ast }VT^{-1}\right)
Tx,Tx\right\rangle }{\left\Vert Tx\right\Vert ^{2}}.
\end{align*}%
If we put 
\begin{equation*}
u=\frac{Tx}{\left\Vert Tx\right\Vert }\neq 0,
\end{equation*}%
then $\left\Vert u\right\Vert =1$ and%
\begin{equation*}
t_{A,B}=\left\langle \left( \left( T^{\ast }\right) ^{-1}V^{\ast
}VT^{-1}\right) u,u\right\rangle \in \left[ m^{2},M^{2}\right] \subset 
\mathring{I}.
\end{equation*}%
By taking $t=t_{A,B}$ in (\ref{e.2.1}) we get (\ref{e.2.6}).

The inequality (\ref{e.2.6}) is equivalent to%
\begin{align*}
\left\langle \circledcirc _{\Phi }\left( V,T\right) y,y\right\rangle & \geq
\Phi \left( \frac{\left\Vert Vx\right\Vert ^{2}}{\left\Vert Tx\right\Vert
^{2}}\right) \left\langle \left\vert T\right\vert ^{2}y,y\right\rangle \\
& +\varphi \left( \frac{\left\Vert Vx\right\Vert ^{2}}{\left\Vert
Tx\right\Vert ^{2}}\right) \left( \left\langle \left\vert V\right\vert
^{2}y,y\right\rangle -\frac{\left\Vert Vx\right\Vert ^{2}}{\left\Vert
Tx\right\Vert ^{2}}\left\langle \left\vert T\right\vert ^{2}y,y\right\rangle
\right)
\end{align*}%
for any $y\in H.$

It can be written as 
\begin{align}
\left\langle \circledcirc _{\Phi }\left( V,T\right) y,y\right\rangle & \geq
\Phi \left( \frac{\left\Vert Vx\right\Vert ^{2}}{\left\Vert Tx\right\Vert
^{2}}\right) \left\Vert Ty\right\Vert ^{2}  \label{e.2.8} \\
& +\varphi \left( \frac{\left\Vert Vx\right\Vert ^{2}}{\left\Vert
Tx\right\Vert ^{2}}\right) \left( \left\Vert Vy\right\Vert ^{2}-\frac{%
\left\Vert Vx\right\Vert ^{2}}{\left\Vert Tx\right\Vert ^{2}}\left\Vert
Ty\right\Vert ^{2}\right) .  \notag
\end{align}

This is an inequality of interest in itself.

In particular, if we take in (\ref{e.2.8}) $y=x,$ then we get the desired
result (\ref{e.2.7}).
\end{proof}

\begin{corollary}
\label{c.2.2}With the assumptions of Theorem \ref{t.2.1}, we have%
\begin{align}
& \circledcirc _{\Phi }\left( V,T\right)  \label{e.2.9} \\
& \geq 2\left( \frac{1}{M^{2}-m^{2}}\int_{m^{2}}^{M^{2}}\Phi \left( t\right)
dt\right) \left\vert T\right\vert ^{2}  \notag \\
& -\frac{1}{M^{2}-m^{2}}\left[ \Phi \left( M^{2}\right) \left(
M^{2}\left\vert T\right\vert ^{2}-\left\vert V\right\vert ^{2}\right) +\Phi
\left( m^{2}\right) \left( \left\vert V\right\vert ^{2}-m\left\vert
T\right\vert ^{2}\right) \right] .  \notag
\end{align}
\end{corollary}

\begin{proof}
If we take the integral mean in the interval $\left[ m^{2},M^{2}\right] $ of
the inequality (\ref{e.2.1}), then we get%
\begin{align}
\circledcirc _{\Phi }\left( V,T\right) & \geq \left( \frac{1}{M^{2}-m^{2}}%
\int_{m^{2}}^{M^{2}}\Phi \left( t\right) dt\right) \left\vert T\right\vert
^{2}  \label{e.2.10} \\
& +\left( \frac{1}{M^{2}-m^{2}}\int_{m^{2}}^{M^{2}}\varphi \left( t\right)
dt\right) \left\vert V\right\vert ^{2}  \notag \\
& -\left( \frac{1}{M^{2}-m^{2}}\int_{m^{2}}^{M^{2}}t\varphi \left( t\right)
dt\right) \left\vert T\right\vert ^{2}.  \notag
\end{align}%
Observe that, since $\varphi \in \partial \Phi ,$ hence%
\begin{equation*}
\frac{1}{M^{2}-m^{2}}\int_{m^{2}}^{M^{2}}\varphi \left( t\right) dt=\frac{%
\Phi \left( M^{2}\right) -\Phi \left( m^{2}\right) }{M^{2}-m^{2}}
\end{equation*}%
and%
\begin{align*}
\frac{1}{M^{2}-m^{2}}\int_{m^{2}}^{M^{2}}t\varphi \left( t\right) dt& =\frac{%
1}{M^{2}-m^{2}}\left[ \left. t\Phi \left( t\right) \right\vert
_{m^{2}}^{M^{2}}-\int_{m^{2}}^{M^{2}}\Phi \left( t\right) dt\right] \\
& =\frac{M^{2}\Phi \left( M^{2}\right) -m^{2}\Phi \left( m^{2}\right) }{%
M^{2}-m^{2}}-\frac{1}{M^{2}-m^{2}}\int_{m^{2}}^{M^{2}}\Phi \left( t\right) dt
\end{align*}%
and by (\ref{e.2.10}) we get%
\begin{align*}
\circledcirc _{\Phi }\left( V,T\right) & \geq \left( \frac{1}{M^{2}-m^{2}}%
\int_{m^{2}}^{M^{2}}\Phi \left( t\right) dt\right) \left\vert T\right\vert
^{2}+\frac{\Phi \left( M^{2}\right) -\Phi \left( m^{2}\right) }{M^{2}-m^{2}}%
\left\vert V\right\vert ^{2} \\
& -\left( \frac{M^{2}\Phi \left( M^{2}\right) -m^{2}\Phi \left( m^{2}\right) 
}{M^{2}-m^{2}}-\frac{1}{M^{2}-m^{2}}\int_{m^{2}}^{M^{2}}\Phi \left( t\right)
dt\right) \left\vert T\right\vert ^{2} \\
& =2\left( \frac{1}{M^{2}-m^{2}}\int_{m^{2}}^{M^{2}}\Phi \left( t\right)
dt\right) \left\vert T\right\vert ^{2} \\
& -\frac{1}{M^{2}-m^{2}}\left[ \Phi \left( M^{2}\right) \left(
M^{2}\left\vert T\right\vert ^{2}-\left\vert V\right\vert ^{2}\right) +\Phi
\left( m^{2}\right) \left( \left\vert V\right\vert ^{2}-m^{2}\left\vert
T\right\vert ^{2}\right) \right]
\end{align*}%
that proves the desired result (\ref{e.2.9}).
\end{proof}

The following result providing upper bounds for the quadratic perspective
also holds.

\begin{theorem}
\label{t.2.2}Let $\Phi :I\rightarrow \mathbb{R}$ be a continuously
differentiable convex function on $\mathring{I}$, $T,$ $V\in \mathcal{B}%
^{-1}\left( H\right) $ such that there exists the positive numbers $m<M$
with $\left[ m^{2},M^{2}\right] \subset \mathring{I}$ satisfying either the
condition (\ref{mM}), or, equivalently, the condition (\ref{mM'}). Then for
any $t\in \mathring{I}$ we have%
\begin{align}
\circledcirc _{\Phi }\left( V,T\right) & \leq \Phi \left( t\right)
\left\vert T\right\vert ^{2}+\circledcirc _{\Phi ^{\prime }\ell }\left(
V,T\right) -t\circledcirc _{\Phi ^{\prime }}\left( V,T\right)  \label{e.2.11}
\\
& \leq \Phi \left( t\right) \left\vert T\right\vert ^{2}+\Phi ^{\prime
}\left( t\right) \left( \left\vert V\right\vert ^{2}-t\left\vert
T\right\vert ^{2}\right)  \notag \\
& +\left[ \Phi _{-}^{\prime }\left( M^{2}\right) -\Phi _{+}^{\prime }\left(
m^{2}\right) \right] \circledcirc _{\left\vert \cdot \right\vert ,t}\left(
V,T\right) ,  \notag
\end{align}%
where $\ell $ is the identity function, i.e. $\ell \left( t\right) =t$ and 
\begin{equation*}
\circledcirc _{\left\vert \cdot \right\vert ,t}\left( V,T\right) :=T^{\ast
}\left\vert \left( T^{\ast }\right) ^{-1}\left( \left\vert V\right\vert
^{2}-t\left\vert T\right\vert ^{2}\right) T^{-1}\right\vert T.
\end{equation*}

In particular, we have%
\begin{align}
& \circledcirc _{\Phi }\left( V,T\right)  \label{e.2.11.a} \\
& \leq \Phi \left( \frac{m^{2}+M^{2}}{2}\right) \left\vert T\right\vert
^{2}+\circledcirc _{\Phi ^{\prime }\ell }\left( V,T\right) -\frac{m^{2}+M^{2}%
}{2}\circledcirc _{\Phi ^{\prime }}\left( V,T\right)  \notag \\
& \leq \Phi \left( \frac{m^{2}+M^{2}}{2}\right) \left\vert T\right\vert
^{2}+\Phi ^{\prime }\left( \frac{m^{2}+M^{2}}{2}\right) \left( \left\vert
V\right\vert ^{2}-\frac{m^{2}+M^{2}}{2}\left\vert T\right\vert ^{2}\right) 
\notag \\
& +\left[ \Phi _{-}^{\prime }\left( M\right) -\Phi _{+}^{\prime }\left(
m\right) \right] \circledcirc _{\left\vert \cdot \right\vert ,\frac{m+M}{2}%
}\left( V,T\right)  \notag \\
& \leq \Phi \left( \frac{m^{2}+M^{2}}{2}\right) \left\vert T\right\vert
^{2}+\Phi ^{\prime }\left( \frac{m^{2}+M^{2}}{2}\right) \left( \left\vert
V\right\vert ^{2}-\frac{m^{2}+M^{2}}{2}\left\vert T\right\vert ^{2}\right) 
\notag \\
& +\frac{1}{2}\left( M^{2}-m^{2}\right) \left[ \Phi _{-}^{\prime }\left(
M^{2}\right) -\Phi _{+}^{\prime }\left( m^{2}\right) \right] \left\vert
T\right\vert ^{2}.  \notag
\end{align}
\end{theorem}

\begin{proof}
By the gradient inequality we have%
\begin{equation}
\Phi ^{\prime }\left( s\right) \left( s-t\right) +\Phi \left( t\right) \geq
\Phi \left( s\right)  \label{e.2.12}
\end{equation}%
for any $s\in \left[ m^{2},M^{2}\right] $ and $t\in \mathring{I}.$

Using the continuous functional calculus for a selfadjoint operator $X$ with 
$\limfunc{Sp}\left( X\right) \subseteq \left[ m^{2},M^{2}\right] \subset 
\mathring{I}$ we have from (\ref{e.2.12}) in the operator order that 
\begin{equation}
\Phi ^{\prime }\left( X\right) \left( X-t1_{H}\right) +\Phi \left( t\right)
1_{H}\geq \Phi \left( X\right)  \label{e.2.13}
\end{equation}%
for any $t\in \mathring{I}$.

Now, if we take $X=\left\vert VT^{-1}\right\vert ^{2}$ in (\ref{e.2.13}) and
since $\limfunc{Sp}\left( \left\vert VT^{-1}\right\vert ^{2}\right)
\subseteq \left[ m^{2},M^{2}\right] ,$ then we get 
\begin{equation}
\Phi ^{\prime }\left( \left\vert VT^{-1}\right\vert ^{2}\right) \left(
\left\vert VT^{-1}\right\vert ^{2}-t1_{H}\right) +\Phi \left( t\right)
1_{H}\geq \Phi \left( \left\vert VT^{-1}\right\vert ^{2}\right)
\label{e.2.14}
\end{equation}%
for any $t\in \mathring{I}$.

So, if we multiply (\ref{e.2.14}) at left with $T^{\ast }$ and at right with 
$T,$ then we get%
\begin{equation}
T^{\ast }\Phi ^{\prime }\left( \left\vert VT^{-1}\right\vert ^{2}\right)
\left( \left\vert VT^{-1}\right\vert ^{2}-t1_{H}\right) T+\Phi \left(
t\right) \left\vert T\right\vert ^{2}\geq T^{\ast }\Phi \left( \left\vert
VT^{-1}\right\vert ^{2}\right) T  \label{e.2.15}
\end{equation}%
for any $t\in \mathring{I}$.

Since%
\begin{equation*}
T^{\ast }\Phi ^{\prime }\left( \left\vert VT^{-1}\right\vert ^{2}\right)
\left( \left\vert VT^{-1}\right\vert ^{2}-t1_{H}\right) T=\circledcirc
_{\Phi ^{\prime }\ell }\left( V,T\right) -t\circledcirc _{\Phi ^{\prime
}}\left( V,T\right) ,
\end{equation*}%
hence by (\ref{e.2.15}) we get the first inequality in (\ref{e.2.11}).

Now, observe that 
\begin{align*}
& T^{\ast }\Phi ^{\prime }\left( \left\vert VT^{-1}\right\vert ^{2}\right)
\left( \left\vert VT^{-1}\right\vert ^{2}-t1_{H}\right) T+\Phi \left(
t\right) \left\vert T\right\vert ^{2} \\
& =T^{\ast }\left( \Phi ^{\prime }\left( \left\vert VT^{-1}\right\vert
^{2}\right) -\Phi ^{\prime }\left( t\right) 1_{H}\right) \left( \left\vert
VT^{-1}\right\vert ^{2}-t1_{H}\right) T+\Phi \left( t\right) \left\vert
T\right\vert ^{2} \\
& +\Phi ^{\prime }\left( t\right) T^{\ast }\left( \left\vert
VT^{-1}\right\vert ^{2}-t1_{H}\right) T+\Phi \left( t\right) \left\vert
T\right\vert ^{2} \\
& =T^{\ast }\left( \Phi ^{\prime }\left( \left\vert VT^{-1}\right\vert
^{2}\right) -\Phi ^{\prime }\left( t\right) 1_{H}\right) \left( \left\vert
VT^{-1}\right\vert ^{2}-t1_{H}\right) T+\Phi \left( t\right) \left\vert
T\right\vert ^{2} \\
& +\Phi ^{\prime }\left( t\right) \left( \left\vert V\right\vert
^{2}-t\left\vert T\right\vert ^{2}\right) +\Phi \left( t\right) \left\vert
T\right\vert ^{2}
\end{align*}%
for any $t\in \mathring{I}$.

Since $\Phi ^{\prime }$ is nondecreasing on $\mathring{I}$ we have for any $%
s\in \left[ m^{2},M^{2}\right] $ and $t\in \mathring{I}$ that 
\begin{align*}
0& \leq \left( \Phi ^{\prime }\left( s\right) -\Phi ^{\prime }\left(
t\right) \right) \left( s-t\right) =\left\vert \left( \Phi ^{\prime }\left(
s\right) -\Phi ^{\prime }\left( t\right) \right) \left( s-t\right)
\right\vert \\
& =\left\vert \Phi ^{\prime }\left( s\right) -\Phi ^{\prime }\left( t\right)
\right\vert \left\vert s-t\right\vert \leq \left[ \Phi _{-}^{\prime }\left(
M\right) -\Phi _{+}^{\prime }\left( m\right) \right] \left\vert
s-t\right\vert ,
\end{align*}%
which, as above, implies in the operator order that%
\begin{align*}
& T^{\ast }\left( \Phi ^{\prime }\left( \left\vert VT^{-1}\right\vert
^{2}\right) -\Phi ^{\prime }\left( t\right) 1_{H}\right) \left( \left\vert
VT^{-1}\right\vert ^{2}-t1_{H}\right) T \\
& \leq \left[ \Phi _{-}^{\prime }\left( M\right) -\Phi _{+}^{\prime }\left(
m\right) \right] \left\vert T^{\ast }\left\vert VT^{-1}\right\vert
^{2}-t1_{H}\right\vert T \\
& =\left[ \Phi _{-}^{\prime }\left( M\right) -\Phi _{+}^{\prime }\left(
m\right) \right] \circledcirc _{\left\vert \cdot \right\vert ,t}\left(
V,T\right) .
\end{align*}%
This proves the second inequality in (\ref{e.2.11}).

We need to prove only the last part of (\ref{e.2.11.a}).

Since $s\in \left[ m^{2},M^{2}\right] ,$ then $\left\vert s-\frac{m^{2}+M^{2}%
}{2}\right\vert \leq \frac{1}{2}\left( M^{2}-m^{2}\right) $ that implies in
the operator order 
\begin{equation*}
\left\vert \left\vert VT^{-1}\right\vert ^{2}-\frac{m^{2}+M^{2}}{2}%
1_{H}\right\vert \leq \frac{1}{2}\left( M^{2}-m^{2}\right) 1_{H},
\end{equation*}%
which by multiplying at left with $T^{\ast }$ and at right with $T$ gives
that 
\begin{equation*}
\circledcirc _{\left\vert \cdot \right\vert ,\frac{m^{2}+M^{2}}{2}}\left(
V,T\right) \leq \frac{1}{2}\left( M^{2}-m^{2}\right) \left\vert T\right\vert
^{2}.
\end{equation*}
\end{proof}

\begin{corollary}
\label{c.2.3}With the assumptions of Theorem \ref{t.2.2}, we have for any $%
x\in H\setminus \left\{ 0\right\} $ that%
\begin{align}
\circledcirc _{\Phi }\left( V,T\right) & \leq \Phi \left( \frac{\left\Vert
Vx\right\Vert ^{2}}{\left\Vert Tx\right\Vert ^{2}}\right) \left\vert
T\right\vert ^{2}+\circledcirc _{\Phi ^{\prime }\ell }\left( V,T\right) -%
\frac{\left\Vert Vx\right\Vert ^{2}}{\left\Vert Tx\right\Vert ^{2}}%
\circledcirc _{\Phi ^{\prime }}\left( V,T\right)  \label{e.2.16} \\
& \leq \Phi \left( \frac{\left\Vert Vx\right\Vert ^{2}}{\left\Vert
Tx\right\Vert ^{2}}\right) \left\vert T\right\vert ^{2}+\Phi ^{\prime
}\left( \frac{\left\Vert Vx\right\Vert ^{2}}{\left\Vert Tx\right\Vert ^{2}}%
\right) \left( \left\vert V\right\vert ^{2}-\frac{\left\Vert Vx\right\Vert
^{2}}{\left\Vert Tx\right\Vert ^{2}}\left\vert T\right\vert ^{2}\right) 
\notag \\
& +\left[ \Phi _{-}^{\prime }\left( M\right) -\Phi _{+}^{\prime }\left(
m\right) \right] \circledcirc _{\left\vert \cdot \right\vert ,\frac{%
\left\Vert Vx\right\Vert ^{2}}{\left\Vert Tx\right\Vert ^{2}}}\left(
V,T\right) .  \notag
\end{align}%
In particular%
\begin{align}
& \left\langle \circledcirc _{\Phi }\left( V,T\right) x,x\right\rangle
\label{e.2.17} \\
& \leq \Phi \left( \frac{\left\Vert Vx\right\Vert ^{2}}{\left\Vert
Tx\right\Vert ^{2}}\right) \left\Vert Tx\right\Vert ^{2}+\left\langle
\circledcirc _{\Phi ^{\prime }\ell }\left( V,T\right) x,x\right\rangle -%
\frac{\left\Vert Vx\right\Vert ^{2}}{\left\Vert Tx\right\Vert ^{2}}%
\left\langle \circledcirc _{\Phi ^{\prime }}\left( V,T\right)
x,x\right\rangle  \notag \\
& \leq \Phi \left( \frac{\left\Vert Vx\right\Vert ^{2}}{\left\Vert
Tx\right\Vert ^{2}}\right) \left\Vert Tx\right\Vert ^{2}+\left[ \Phi
_{-}^{\prime }\left( M\right) -\Phi _{+}^{\prime }\left( m\right) \right]
\left\langle \circledcirc _{\left\vert \cdot \right\vert ,\frac{\left\Vert
Vx\right\Vert ^{2}}{\left\Vert Tx\right\Vert ^{2}}}\left( V,T\right)
x,x\right\rangle  \notag
\end{align}%
for any $x\in H\setminus \left\{ 0\right\} .$
\end{corollary}

If we take the integral mean in the interval $\left[ m^{2},M^{2}\right] $ of
the inequality (\ref{e.2.11}) we can also state the following result.

\begin{corollary}
\label{c.2.4}With the assumptions of Theorem \ref{t.2.2}, we have
\end{corollary}

\begin{align}
& \circledcirc _{\Phi }\left( V,T\right)  \label{e.2.18} \\
& \leq \left( \frac{1}{M^{2}-m^{2}}\int_{m^{2}}^{M^{2}}\Phi \left( t\right)
dt\right) \left\vert T\right\vert ^{2}+\circledcirc _{\Phi ^{\prime }\ell
}\left( V,T\right) -\frac{m^{2}+M^{2}}{2}\circledcirc _{\Phi ^{\prime
}}\left( V,T\right)  \notag \\
& \leq 2\left( \frac{1}{M^{2}-m^{2}}\int_{m^{2}}^{M^{2}}\Phi \left( t\right)
dt\right) \left\vert T\right\vert ^{2}  \notag \\
& -\frac{1}{M^{2}-m^{2}}\left[ \Phi \left( M^{2}\right) \left(
M^{2}\left\vert T\right\vert ^{2}-\left\vert V\right\vert ^{2}\right) +\Phi
\left( m^{2}\right) \left( \left\vert V\right\vert ^{2}-m^{2}\left\vert
T\right\vert ^{2}\right) \right]  \notag \\
& +\left[ \Phi _{-}^{\prime }\left( M^{2}\right) -\Phi _{+}^{\prime }\left(
m^{2}\right) \right] \frac{1}{M^{2}-m^{2}}\int_{m^{2}}^{M^{2}}\circledcirc
_{\left\vert \cdot \right\vert ,t}\left( V,T\right) dt.  \notag
\end{align}

\section{Applications for Quadratic Weighted Geometric Mean}

For $x\not=y$ and $p\in \mathbb{R}\setminus \{-1,0\}$, we define the $p$-%
\textit{logarithmic mean} (\textit{generalized logarithmic mean}) $%
L_{p}(x,y) $ by 
\begin{equation*}
L_{p}(x,y):=\left[ \frac{y^{p+1}-x^{p+1}}{(p+1)(y-x)}\right] ^{1/p}.
\end{equation*}%
In fact the singularities at $p=-1,$ $0$ are removable and $L_{p}$ can be
defined for $p=-1,$ $0$ so as to make $L_{p}(x,y)$ a continuous function of $%
p$. In the limit as $p\rightarrow 0$ we obtain the \textit{identric mean }$%
I(x,y)$, given by 
\begin{equation}
I(x,y):=\frac{1}{e}\ \left( \frac{y^{y}}{x^{x}}\right) ^{1/(y-x)},
\label{Id}
\end{equation}%
and in the case $p\rightarrow -1$ the \textit{logarithmic mean} $L(x,y)$,
given by 
\begin{equation*}
L(x,y):=\frac{y-x}{\ln y-\ln x}.
\end{equation*}%
In each case we define the mean as $x$ when $y=x$, which occurs as the
limiting value of $L_{p}(x,y)$ for $y\rightarrow x$.

If we consider the continuous function $f_{\nu }:[0,\infty )\rightarrow
\lbrack 0,\infty )$, $f_{\nu }\left( t\right) =t^{\nu }$ then the \textit{%
quadratic} \textit{weighted operator geometric mean }can be interpreted as
the quadratic perspective $\circledcirc _{f_{\nu }}\left( B,A\right) $ of $%
T, $ $V\in \mathcal{B}^{-1}\left( H\right) $ and $f_{\nu },$ namely%
\begin{equation*}
\circledcirc _{f_{\nu }}\left( V,T\right) =T\circledS _{\nu }V.
\end{equation*}%
Consider the convex function $f=-f_{\nu }.$ Then by applying the
inequalities (\ref{e.2.1}) and (\ref{e.2.1.a})\ we have%
\begin{equation}
T\circledS _{\nu }V\leq \left( 1-\nu \right) t^{\nu }\left\vert T\right\vert
^{2}+\nu t^{\nu -1}\left\vert V\right\vert ^{2}=\left( t^{\nu }\left\vert
T\right\vert ^{2}\right) \nabla _{\nu }\left( t^{\nu -1}\left\vert
V\right\vert ^{2}\right) ,  \label{e.3.1}
\end{equation}%
for any $t>0$ and $\nu \in \left[ 0,1\right] ,$ and 
\begin{equation}
T\circledS _{\nu }V\leq \left( 1-\nu \right) \left( \frac{m^{2}+M^{2}}{2}%
\right) ^{\nu }\left\vert T\right\vert ^{2}+\nu \left( \frac{m^{2}+M^{2}}{2}%
\right) ^{\nu -1}\left\vert V\right\vert ^{2}  \label{e.3.2}
\end{equation}%
for any $\nu \in \left[ 0,1\right] ,$ provided either the condition (\ref{mM}%
), or, equivalently, the condition (\ref{mM'}) is valid.

From (\ref{e.2.6}) and (\ref{e.2.7}) we have for any $x\in H\setminus
\left\{ 0\right\} $ and $\nu \in \left[ 0,1\right] $ that%
\begin{equation}
T\circledS _{\nu }V\leq \left( 1-\nu \right) \left( \frac{\left\Vert
Vx\right\Vert ^{2}}{\left\Vert Tx\right\Vert ^{2}}\right) ^{\nu }\left\vert
T\right\vert ^{2}+\nu \left( \frac{\left\Vert Tx\right\Vert ^{2}}{\left\Vert
Vx\right\Vert ^{2}}\right) ^{1-\nu }\left\vert V\right\vert ^{2}
\label{e.3.3}
\end{equation}%
and 
\begin{equation}
\left\langle T\circledS _{\nu }Vx,x\right\rangle \leq \left\Vert
Tx\right\Vert ^{2\left( 1-\nu \right) }\left\Vert Vx\right\Vert ^{2\nu },
\label{e.3.4}
\end{equation}%
for any $\nu \in \left[ 0,1\right] .$

The inequality (\ref{e.1.6}) can be written as 
\begin{equation}
\left\langle T\circledS _{\nu }Vx,x\right\rangle \leq \left( 1-\nu \right)
\left\Vert Tx\right\Vert ^{2}+\nu \left\Vert Vx\right\Vert ^{2}  \label{YY}
\end{equation}%
for any $x\in H.$

By utilizing the scalar arithmetic mean-geometric mean inequality we also
have 
\begin{equation}
\left\Vert Tx\right\Vert ^{2\left( 1-\nu \right) }\left\Vert Vx\right\Vert
^{2\nu }\leq \left( 1-\nu \right) \left\Vert Tx\right\Vert ^{2}+\nu
\left\Vert Vx\right\Vert ^{2}  \label{e.3.4.a}
\end{equation}%
for any $x\in H.$

Therefore by (\ref{e.3.4}) and (\ref{e.3.4.a}) we have the following vector
inequality improving (\ref{YY})%
\begin{equation}
\left\langle T\circledS _{\nu }Vx,x\right\rangle \leq \left\Vert
Tx\right\Vert ^{2\left( 1-\nu \right) }\left\Vert Vx\right\Vert ^{2\nu }\leq
\left( 1-\nu \right) \left\Vert Tx\right\Vert ^{2}+\nu \left\Vert
Vx\right\Vert ^{2}  \label{e.3.4.b}
\end{equation}%
for any $x\in H.$

From (\ref{e.2.9}) we have%
\begin{align}
T\circledS _{\nu }V& \leq 2L_{\nu }^{\nu }(m^{2},M^{2})\left\vert
T\right\vert ^{2}  \label{e.3.5} \\
& -\frac{1}{M^{2}-m^{2}}\left[ M^{2\nu }\left( M^{2}\left\vert T\right\vert
^{2}-\left\vert V\right\vert ^{2}\right) +m^{2\nu }\left( \left\vert
V\right\vert ^{2}-m^{2}\left\vert T\right\vert ^{2}\right) \right]  \notag
\end{align}%
for any $\nu \in \left( 0,1\right) ,$ provided either the condition (\ref{mM}%
), or, equivalently, the condition (\ref{mM'}) is valid.

If $T,$ $V\in \mathcal{B}^{-1}\left( H\right) $ satisfy the condition (\ref%
{mM}), then by (\ref{e.2.11.a}) we have%
\begin{align}
T\circledS _{\nu }V& \geq \left( \frac{m^{2}+M^{2}}{2}\right) ^{\nu
}\left\vert T\right\vert ^{2}+\nu T\circledS _{\nu }V-\nu \frac{m^{2}+M^{2}}{%
2}T\circledS _{\nu -1}V  \label{e.3.7} \\
& \geq \left( 1-\nu \right) \left( \frac{m^{2}+M^{2}}{2}\right) ^{\nu
}\left\vert T\right\vert ^{2}+\nu \left( \frac{m^{2}+M^{2}}{2}\right) ^{\nu
-1}\left\vert V\right\vert ^{2}  \notag \\
& +\nu \left( M^{2\left( \nu -1\right) }-m^{2\left( \nu -1\right) }\right)
\circledcirc _{\left\vert \cdot \right\vert ,\frac{m+M}{2}}\left( V,T\right)
\notag \\
& \geq \left( 1-\nu \right) \left( \frac{m^{2}+M^{2}}{2}\right) ^{\nu
}\left\vert T\right\vert ^{2}+\nu \left( \frac{m^{2}+M^{2}}{2}\right) ^{\nu
-1}\left\vert V\right\vert ^{2}  \notag \\
& +\frac{1}{2}\nu \left( M^{2}-m^{2}\right) \left( M^{2\left( \nu -1\right)
}-m^{2\left( \nu -1\right) }\right) .  \notag
\end{align}%
From the last inequality in (\ref{e.3.7}) we get\ 
\begin{align}
& \frac{1}{2}\nu \left( M^{2}-m^{2}\right) \left( \frac{M^{2\left( 1-\nu
\right) }-m^{2\left( 1-\nu \right) }}{m^{2\left( 1-\nu \right) }M^{2\left(
1-\nu \right) }}\right)  \label{e.3.8} \\
& \geq \left( 1-\nu \right) \left( \frac{m^{2}+M^{2}}{2}\right) ^{\nu
}\left\vert T\right\vert ^{2}+\nu \left( \frac{m^{2}+M^{2}}{2}\right) ^{\nu
-1}\left\vert V\right\vert ^{2}-T\circledS _{\nu }V\geq 0,  \notag
\end{align}%
for any $\nu \in \left[ 0,1\right] ,$ which provides a simple reverse for (%
\ref{e.3.2}).

\section{Applications for Quadratic Relative Operator Entropy}

Consider the logarithmic function $\ln .$ Then the quadratic relative
operator entropy can be interpreted as the perspective of $\ln $, namely%
\begin{equation*}
\circledcirc _{\ln }\left( V,T\right) =\odot \left( T|V\right) :=T^{\ast
}\ln \left( \left\vert VT^{-1}\right\vert ^{2}\right) T
\end{equation*}%
provided $T,$ $V\in \mathcal{B}^{-1}\left( H\right) .$

If we use the inequalities (\ref{e.2.1}) and (\ref{e.2.1.a}) for the convex
function $f=-\ln $ we have%
\begin{equation}
\odot \left( T|V\right) \leq \left( \ln t\right) \left\vert T\right\vert
^{2}-\left\vert T\right\vert ^{2}+t^{-1}\left\vert V\right\vert ^{2},
\label{e.6.2}
\end{equation}%
for any $t>0$ and $T,$ $V\in \mathcal{B}^{-1}\left( H\right) .$

In particular, if $T,$ $V$ satisfy the condition (\ref{mM}), then 
\begin{align}
\odot \left( T|V\right) & \leq \left[ \ln \left( \frac{m^{2}+M^{2}}{2}%
\right) \right] \left\vert T\right\vert ^{2}  \label{e.6.3} \\
& +\left( \frac{m^{2}+M^{2}}{2}\right) ^{-1}\left( \left\vert V\right\vert
^{2}-\frac{m^{2}+M^{2}}{2}\left\vert T\right\vert ^{2}\right) .  \notag
\end{align}%
From the inequalities (\ref{e.2.6}) and (\ref{e.2.7}) we have%
\begin{equation}
\odot \left( T|V\right) \leq \ln \left( \frac{\left\Vert Vx\right\Vert ^{2}}{%
\left\Vert Tx\right\Vert ^{2}}\right) \left\vert T\right\vert ^{2}+\frac{%
\left\Vert Tx\right\Vert ^{2}}{\left\Vert Vx\right\Vert ^{2}}\left\vert
V\right\vert ^{2}-\left\vert T\right\vert ^{2}  \label{e.6.4}
\end{equation}%
and 
\begin{equation}
\left\langle \odot \left( T|V\right) x,x\right\rangle \leq \left\Vert
Tx\right\Vert ^{2}\ln \left( \frac{\left\Vert Vx\right\Vert ^{2}}{\left\Vert
Tx\right\Vert ^{2}}\right) ,  \label{e.6.5}
\end{equation}%
for any $x\in H,$ $x\neq 0.$

The following inequality for the relative operator entropy is known, see (%
\ref{e.1.11})%
\begin{equation}
\odot \left( T|V\right) \leq \left\vert V\right\vert ^{2}-\left\vert
T\right\vert ^{2}  \label{e.6.5.a}
\end{equation}%
for any $T,$ $V\in \mathcal{B}^{-1}\left( H\right) .$

This inequality is equivalent to 
\begin{equation}
\left\langle \odot \left( T|V\right) x,x\right\rangle \leq \left\Vert
Vx\right\Vert ^{2}-\left\Vert Tx\right\Vert ^{2}  \label{e.6.5.b}
\end{equation}%
for any $x\in H.$

We know the following elementary inequality that holds for the logarithm%
\begin{equation*}
\ln t\leq t-1\text{ for any\textit{\ }}t>0.
\end{equation*}%
If we take in this inequality $t=\frac{\left\Vert Vx\right\Vert ^{2}}{%
\left\Vert Tx\right\Vert ^{2}}>0,$ $x\in H,$ $x\neq 0$ and multiply with $%
\left\Vert Tx\right\Vert ^{2}>0,$ then we get%
\begin{equation}
\left\Vert Tx\right\Vert ^{2}\ln \left( \frac{\left\Vert Vx\right\Vert ^{2}}{%
\left\Vert Tx\right\Vert ^{2}}\right) \leq \left\Vert Vx\right\Vert
^{2}-\left\Vert Tx\right\Vert ^{2}  \label{e.6.5.c}
\end{equation}%
for any $x\in H,$ $x\neq 0.$

Therefore, by (\ref{e.6.5}) and (\ref{e.6.5.c}) we have 
\begin{equation*}
\left\langle \odot \left( T|V\right) x,x\right\rangle \leq \left\Vert
Tx\right\Vert ^{2}\ln \left( \frac{\left\Vert Vx\right\Vert ^{2}}{\left\Vert
Tx\right\Vert ^{2}}\right) \leq \left\Vert Vx\right\Vert ^{2}-\left\Vert
Tx\right\Vert ^{2}
\end{equation*}%
for any $x\in H,$ $x\neq 0$ that is an improvement of (\ref{e.6.5.b}).

From (\ref{e.2.9}) we also have 
\begin{align}
\odot \left( T|V\right) & \leq 2\left[ \ln I\left( m^{2},M^{2}\right) \right]
\left\vert T\right\vert ^{2}  \label{e.6.6} \\
& -\frac{1}{M^{2}-m^{2}}\left[ \ln M^{2}\left( M^{2}\left\vert T\right\vert
^{2}-\left\vert V\right\vert ^{2}\right) +\ln m^{2}\left( \left\vert
V\right\vert ^{2}-m^{2}\left\vert T\right\vert ^{2}\right) \right] ,  \notag
\end{align}%
where $I\left( \cdot ,\cdot \right) $ is the identric mean defined in (\ref%
{Id}) and 
\begin{equation*}
\frac{1}{M^{2}-m^{2}}\int_{m^{2}}^{M^{2}}\ln tdt=\ln I\left(
m^{2},M^{2}\right) .
\end{equation*}

From (\ref{e.2.11.a}) we also have%
\begin{align}
& \odot \left( T|V\right)  \label{e.6.8} \\
& \geq \left[ \ln \left( \frac{m^{2}+M^{2}}{2}\right) \right] \left\vert
T\right\vert ^{2}+\left\vert T\right\vert ^{2}-\frac{m^{2}+M^{2}}{2}%
\left\vert T\right\vert ^{2}\left\vert V^{\ast }\right\vert ^{-2}\left\vert
T\right\vert ^{2}  \notag \\
& \geq \left[ \ln \left( \frac{m^{2}+M^{2}}{2}\right) \right] \left\vert
T\right\vert ^{2}+\left( \frac{m^{2}+M^{2}}{2}\right) ^{-1}\left( \left\vert
V\right\vert ^{2}-\frac{m^{2}+M^{2}}{2}\left\vert T\right\vert ^{2}\right) 
\notag \\
& -\frac{M^{2}-m^{2}}{m^{2}M^{2}}\circledcirc _{\left\vert \cdot \right\vert
,\frac{m+M}{2}}\left( T|V\right)  \notag \\
& \geq \left[ \ln \left( \frac{m^{2}+M^{2}}{2}\right) \right] \left\vert
T\right\vert ^{2}+\left( \frac{m^{2}+M^{2}}{2}\right) ^{-1}\left( \left\vert
V\right\vert ^{2}-\frac{m^{2}+M^{2}}{2}\left\vert T\right\vert ^{2}\right) 
\notag \\
& -\frac{1}{2}\frac{\left( M^{2}-m^{2}\right) ^{2}}{m^{2}M^{2}},  \notag
\end{align}%
provided $T,$ $V\in \mathcal{B}^{-1}\left( H\right) $ satisfying the
condition (\ref{mM}).

From the last part of (\ref{e.6.8}) we get%
\begin{align}
\frac{1}{2}\frac{\left( M^{2}-m^{2}\right) ^{2}}{m^{2}M^{2}}& \geq \left[
\ln \left( \frac{m^{2}+M^{2}}{2}\right) \right] \left\vert T\right\vert ^{2}
\label{e.6.9} \\
& +\left( \frac{m^{2}+M^{2}}{2}\right) ^{-1}\left( \left\vert V\right\vert
^{2}-\frac{m^{2}+M^{2}}{2}\left\vert T\right\vert ^{2}\right) -\odot \left(
T|V\right)  \notag \\
& \geq 0  \notag
\end{align}%
that provides a simple reverse of (\ref{e.6.3}).

If one considers the convex function $f\left( t\right) =t\ln t$ for $t>0,$
that one can get other logarithmic inequalities as above. The details are
left to the interested reader.

\end{document}